\documentclass{amsart}

\theoremstyle{definition}

\theoremstyle{remark}

\numberwithin{equation}{section}

    \newcommand{\BA}{{\mathbb {A}}} 
    \newcommand{\BC}{{\mathbb {C}}} 
     
    \newcommand{\BG}{{\mathbb {G}}}

    \newcommand{\CG}{{\mathcal {G}}}

    \newcommand{\CS}{{\mathcal {S}}}

    \newcommand{\tr}{{\mathrm{tr}}}

    \newcommand{\Vol}{{\mathrm{Vol}}}

    \newcommand{\pair}[1]{\langle {#1} \rangle}

    \newcommand{\ra}{\rightarrow} 
    \newcommand{\bs}{\backslash}

    \theoremstyle{plain}
    \newtheorem{thm}{Theorem}[section] \newtheorem{cor}[thm]{Corollary}
    \newtheorem{lem}[thm]{Lemma}  \newtheorem{prop}[thm]{Proposition}

    \numberwithin{equation}{section}

\begin{document}

\title{A Siegel-Weil formula for $(U(1,1), U(V))$ over a function field with $\dim V$ greater than 2}

\author{Wei Xiong}
\address{College of Mathematics and Econometrics, Hunan University,
Changsha 410082, China}

\email{weixiong@amss.ac.cn}
\thanks{This work was supported by the National Natural Science Foundation of China (Grant No. 11301164 and No. 11671380) and the Fundamental Research Funds for the Central Universities of China.}


\date{}

\dedicatory{To the memory of Professor Linsheng Yin.}


\begin{abstract}
We establish a Siegel-Weil formula for the dual pair $(U(1,1), U(V))$ over a function field, where $V$ is a hermitian space of dimension greater than 2.
\end{abstract}

\maketitle

\section*{Introduction}
Let $F$ be a function field in one variable over a finite field with $\mathrm{char}(F)\neq 2$, and let $E$ be a quadratic field extension of $F$. We denote by $\BA$ the adele ring of $F$, and let $\psi$ be a nontrivial character of $F\bs \BA$. Let $V$ be a non-degenerate hermitian space over $E$, of dimension $m\geq 3$, with hermitian form $(,): V\times V \ra E$. Let $H=U(V)$ be the associated unitary group.

In this paper, we establish a Siegel-Weil formula for the unitary dual pair $(U(1,1), U(V))$. A particular form of this formula states that for any Schwartz-Bruhat function $\Phi \in \CS(V(\BA))$, the following identity holds:
\[
\int_{H(F)\bs H(\BA)}\left(\sum_{\xi \in V(F)}\Phi(h^{-1}\xi) \right)dh=2\Phi(0)+2\sum_{b\in F}\int_{V(\BA)}\Phi(x)\psi(b(x,x))dx.
\]
where $dh$ is the Tamagawa measure on $H(\BA)$ and $dx$ is the self-dual Haar measure on $V(\BA)$ with respect to $\psi$.

This work is motivated by Haris's pioneering work \cite{Haris}, in which he established a Siegel-Weil formula for orthogonal groups over a function field. In proving the formula, we use Haris's results in \cite{Haris} heavily.

The Siegel-Weil formulas over number fields have been well investigated (see for example \cite{KR1, KR2, KR3, Ichino04, Ichino07, Yamana, GQT} and the references cited therein). In contrast, the Siegel-Weil formulas over function fields are less studied and are only known in some special cases (see for example \cite{Haris, Wei}).\\

{\bf Acknowledgments}.
I thank Professor Ye Tian and Professor Song Wang for inviting me to visit the Morningside Center of Mathematics and for many helpful discussions.

This paper is dedicated to the memory of Professor Linsheng Yin. The author benefited greatly from Professor Yin's guidance as a post-doc at the Tsinghua University. Professor Yin was an expert on function fields and I hope that it is suitable to dedicate this work to his memory.

\section{Notations and statement of main result}
Recall that $F$ is a function field in one variable over a finite field with $\mathrm{char}(F)\neq 2$, and $E$ is a quadratic field extension of $F$. We denote the adele ring of $F$ by $\BA$, and denote the adele ring of $E$ by $\BA_E$.

Let $V=E^m$ be the space of column vectors over $E$ of dimension $m$, equipped with a non-degenerate hermitian form $(,): V\times V \ra E$. We always assume $m\geq 3$.

Let $H=U(V)$ be the unitary group associated with $V$, given by
\[
H(F)=\{h\in GL_m(E): (hx, hy)=(x,y), \forall x,y\in V\}.
\]
Let $G=U(1,1)$ be the quasi-split unitary group of rank 1, given by
\[
G(F)=\{g\in GL_2(E): g \begin{pmatrix} 0 &1\\-1& 0 \end{pmatrix} {}^t\!\bar{g}=\begin{pmatrix} 0 &1\\-1& 0 \end{pmatrix}\}.
\]
Then $(G, H)$ forms a unitary dual pair.

The group $G$ has a Siegel parabolic subgroup $P=NM$, where $N$ is the unipotent radical and $M$ is the Levi subgroup. Precisely,
\[
N(F)=\{n(b)=\begin{pmatrix} 1 &b\\0&1 \end{pmatrix} :b\in F\}
\]
and
\[
M(F)=\{m(a)=\begin{pmatrix} a &0\\0&\bar{a}^{-1} \end{pmatrix} :a\in E^{\times}\}.
\]
Then we have the Bruhat decomposition $G(F)=P(F)\cup P(F)wN(F)$, where $w=\begin{pmatrix} 0 &1\\-1& 0 \end{pmatrix}$.

For a place $v$ of $F$, let $F_v$ be the completion of $F$ at $v$ and let $O_v$ be the ring of integers of $F_v$. Let $K=\prod_v G(O_v)$. Then there is the Iwasawa decomposition
\[
G(\BA)=P(\BA)\cdot K.
\]
For $g\in G(\BA)$, we write $g=n(b)m(a)k$ with $n(b)\in N(\BA)$, $m(a)\in M(\BA)$ and $k\in K$, and we put
\[
|a(g)|=|a|_{\BA_E}.
\]

Fix a nontrivial character $\psi: F\bs \BA \ra \BC^{\times}$, and fix a Hecke character $\chi: E^{\times}\bs \BA_E^{\times} \ra \BC^{\times}$ such that $\chi|_{\BA^{\times}}=\epsilon$, where $\epsilon$ is the quadratic character associated to the quadratic extension $E/F$ by class field theory. Then there is an associated Weil representation $\omega=\omega_{\psi, \chi}$ of $G(\BA)\times H(\BA)$ on the space $\CS(V(\BA))$ of Schwartz-Bruhat functions on $V(\BA)$, given by
\[\begin{aligned}
& \omega(m(a))\Phi(x)=\chi(a)|a|_{\BA_E}^{\frac{m}{2}}\Phi(xa),\\
& \omega(n(b))\Phi(x)=\psi(b(x,x))\Phi(x),\\
& \omega(w)\Phi(x)=\int_{V(\BA)}\Phi(y)\psi(\tr_{E/F}(y,x))dy,\\
& \omega(h)\Phi(x)=\Phi(h^{-1}x),
\end{aligned}\]
for $\Phi \in \CS(V(\BA))$, $x\in V(\BA)$, $m(a)\in M(\BA)$, $n(b)\in N(\BA)$, and $h\in H(\BA)$, where $dy$ is the self-dual Haar measure on $V(\BA)$ with respect to the pairing on $V(\BA)$ given by $(x,y)\mapsto \psi(\tr_{E/F}(x,y))$ (see \cite[Lem. 4.1]{KS} for the local case and see \cite[$\S 1$]{Ichino04} for the number field case).

To each $\Phi \in \CS(V(\BA)$, there are two associated objects, which are both functions on $G(\BA)$. One is the Siegel Eisenstein series and the other is the theta integral. The Siegel-Weil formula relates these two objects. Now we define them.

For $\Phi \in \CS(V(\BA))$ and $s\in \BC$, the Siegel Eisenstein series on $G(\BA)$ is defined by
\[
E(g,s,\Phi)=\sum_{\gamma \in P(F)\bs G(F)} f_{\Phi}^{(s)}(\gamma g), \quad \forall g\in G(\BA),
\]
where $f_{\Phi}^{(s)}(g)=|a(g)|^{s-s_0}\omega(g)\Phi(0)$ is a function on $G(\BA)$ and $s_0=(m-1)/2$.

On the other hand, for $\Phi \in \CS(V(\BA))$, the theta integral on $G(\BA)$ is defined by
\[
I(g,\Phi)=\int_{[H]} \Theta(\Phi)(g,h)dh, \quad \forall g\in G(\BA),
\]
where $\Theta(\Phi)(g,h)=\sum_{x \in V(F)}\omega(g,h)\Phi(x)$ is the theta kernel function, $[H]:=H(F)\bs H(\BA)$ is the quotient group, and $dh$ is the Tamagawa measure on $H(\BA)$.

Note that for a connected algebraic group $\CG$ over $F$, there is a canonical Haar measure on $\CG(\BA)$ called the Tamagawa measure (see \cite{Weil65, Weil82} and \cite{Oes}), which is derived from a left invariant gauge form on $\CG$ (a gauge form on $\CG$ is a nonzero algebraic differential form on $\CG$ defined over $F$ and of degree $\dim \CG$ ). Note that Weil in \cite{Weil65, Weil82} defined the Tamagawa measure for any system of convergence factors for $\CG$, while Oesterl\'{e} in \cite{Oes} specified a canonical system of convergence factors (i.e. the local Artin $L$-values $L_v(1, \CG)$).

Now we state the main result of this paper.
\begin{thm}
Assume $m\geq 3$. Then for any $\Phi \in \CS(V(\BA))$ and $g\in G(\BA)$, $E(g,s,\Phi)$ is holomorphic at $s_0=(m-1)/2$, $I(g, \Phi)$ is absolutely convergent, and
\[
I(g,\Phi)=2E(g, s_0, \Phi).
\]
\end{thm}

To prove this result we follow the method in \cite{Haris} (also \cite{Igusa}). We give the main ideas below. First consider the case where $g=1$ is the identity element in $G(\BA)$.

By Bruhat decomposition $G(F)=P(F)\bigcup P(F)wN(F)$, we have
\[
E(1,s_0,\Phi)=\Phi(0)+\sum_{b\in F}\int_{V(\BA)}\Phi(x)\psi(b(x,x))dx.
\]
On the other hand, by the orbit decomposition $V(F)= \{0\}\cup \bigcup_{a\in F}V_a(F)$ for the action of $H(F)$ on $V(F)$, where $V_a(F)=\{x\in V(F): x\neq 0, q(x)=a\}$, we have
\[
I(1, \Phi)=\tau(H)\Phi(0)+\sum_{a\in F} \int_{[H]}\left(\sum_{\xi \in V_a(F)}\Phi(h^{-1}\xi)\right)dh,
\]
where $\tau(H)=\Vol([H], dh)$ is the Tamagawa number of $H$. In fact, $\tau(H)=2$.

For $\Phi \in \CS(V(\BA))$, define a function $F_{\Phi}^{*}$ on $\BA$  by
\[
F_{\Phi}^{*}(a)=\int_{V(\BA)}\Phi(x)\psi(a(x,x))dx.
\]
Then
\[
E(1,s_0,\Phi)=\Phi(0)+\sum_{a\in F}F_{\Phi}^{*}(a).
\]
Let $F_{\Phi}$ be the Fourier transform of $F_{\Phi}^{*}$ with respect to $\psi$. Then we can show that
\[
F_{\Phi}(a)=\frac{1}{\tau(H_a)} \int_{[H]}\left(\sum_{\xi\in V_a(F)}\Phi(h^{-1}\xi)\right)dh,
\]
where $H_a$ is the stabilizer of $H$ at any $\xi_a \in V_a(F)$ when $V_a(F)\neq \emptyset$ (if $V_a(F)=\emptyset$, then $F_{\Phi}(a)=0$). In fact, we also have $\tau(H_a)=2$. Then
\[
I(1,\Phi)=\tau(H)\Phi(0)+\sum_{a\in F} \tau(H_a)F_{\Phi}(a)=2\Phi(0)+2\sum_{a\in F} F_{\Phi}(a).
\]
Now consider the mapping $q: V\ra F$ given by $q(x)=(x,x)$. Then $(V, q)$ is a quadratic space over $F$ of dimension $2m$. Consider the adelization $q_{\BA}: V(\BA) \ra \BA$ given by $q_{\BA}(x)=(x,x)$ for $x\in V(\BA)$. A key result in this paper is that $q_{\BA}$ satisfies a so-called ``condition (B)'' defined by Weil in \cite[Prop. 2]{Weil65}, and hence we have
\[
\sum_{a\in F}F_{\Phi}^{*}(a)=\sum_{a\in F} F_{\Phi}(a),
\]
and both series are absolutely convergent.
So $I(1,\Phi)=2E(1,s_0, \Phi)$ and both are absolutely convergent.

In general, $E(g,s_0,\Phi)=E(1,s_0,\omega(g)\Phi)$ and $I(g,\Phi)=I(1,\omega(g)\Phi)$.
The desired formula thus follows.

\section{Proof of the formula}
In this section, we give the details of the proof of the Siegel-Weil formula.

\subsection{Some results of Weil}
We first recall some results of Weil in \cite{Weil65}.
\begin{prop}(\cite[Prop. 1]{Weil65})
Let $X$ and $G$ be locally compact abelian groups, and let $f: X \ra G$ be a continuous mapping. Let $G^{*}$ be the dual group of $G$. Let $dg$ be a Haar measure on $G$, and let $dg^{*}$ be the dual measure. Let $dx$ be a Haar measure on $X$. For $\Phi \in \CS(X)$ a Schwartz-Bruhat function on $X$, define a function on $G^{*}$ by
\[
F_{\Phi}^{*}(g^{*})=\int_X \Phi(x)\pair{f(x), g^{*}}dx.
\]
Suppose $f$ satisfies the following condition:

(A) \quad for $\Phi \in \CS(X)$, $F_{\Phi}^{*}$ is integrable on $G^{*}$ and the integral $\int |F_{\Phi}^{*}|dg^{*}$ converges uniformly on every compact subset of $\CS(X)$.

Then for every $g\in G$, there corresponds a measure $\mu_g$ on $X$, of support contained in $f^{-1}(\{g\})$, such that for every continuous function $\Phi$ on $X$ with compact support, the function on $G$ defined by $F_{\Phi}(g)=\int \Phi d\mu_g$ is continuous and satisfies $\int F_{\Phi}dg=\int \Phi dx$. Moreover, the measures $\mu_g$ are tempered; and for every $\Phi \in \CS(X)$, $F_{\Phi}$ is continuous, belonging to $L^1(G)$, satisfies $\int F_{\Phi}dg=\int \Phi dx$, and is the Fourier transform of $F_{\Phi}^{*}$.
\end{prop}

Weil gave criteria for the above ``condition (A)".
\begin{prop}(\cite[Prop. 7 and the bottom of p. 8]{Weil65})
Let $G$ be a locally compact abelian group, $\Gamma$ a discrete subgroup of $G$ such that $G/\Gamma$ is compact, and $\Gamma_{*}$ the discrete subgroup of $G^{*}$ which corresponds by duality to $\Gamma$. Let $X$ be a locally compact abelian group and $f: X\ra G$ a continuous mapping.

(i) Suppose $f$ satisfies the following condition:

(B)\quad for any $\Phi \in \CS(X)$ and $g^{*}\in G^{*}$, the series
\[
\sum_{\gamma^{*}\in \Gamma_{*}}|F_{\Phi}^{*}(g^{*}+\gamma^{*})|
\]
is convergent, and is uniformly convergent on every compact subset of $\CS(X)\times G^{*}$.

Then $f$ satisfies ``condition (A)" of the above proposition. Moreover, if $F_{\Phi}$ denotes the Fourier transform of the function $F_{\Phi}^{*}$, then
\[
\sum_{\gamma \in \Gamma}F_{\Phi}(\gamma)=\sum_{\gamma^{*}\in \Gamma_{*}}F_{\Phi}^{*}(\gamma^{*}),
\]
the two series are absolutely convergent.\\

(ii) Suppose $f$ satisfies the following two conditions:

($B_0$) $(\Phi, g^{*})\ra \Phi_{g^{*}}$ is a continuous mapping of $\CS(X)\times G^{*}$ into $\CS(X)$, where $\Phi_{g^{*}}$ is given by $\Phi_{g^{*}}(x)=\Phi(x)\pair{f(x), g^{*}}$;

($B_1$) the series $\sum_{\gamma^{*}\in \Gamma_{*}}|F_{\Phi}^{*}(\gamma^{*})|$ is uniformly convergent on every compact subset of $\CS(X)$.

Then $f$ satisfies ``condition (B)", and all the above conclusions hold.
\end{prop}

In the following, we will identify $\BA$ with its dual via the pairing $\pair{a,b}\mapsto \psi(ab)$, and identify $V(\BA)$ with its dual via the pairing $\pair{x,y}\mapsto \psi(\tr_{E/F}(x,y))$. For a place $v$ of $F$, we also identify $F_v$ (resp. $V(F_v)$) with its dual via $\psi_v$ similarly. We will always choose self-dual Haar measures on the groups $\BA$, $F_v$, $V(\BA)$ and $V(F_v)$.

We want to show that the mapping $q_{\BA}: V(\BA) \ra \BA$ given by $q_{\BA}(x)=(x, x)$ satisfies ``conditions ($B_0$) and ($B_1$)'' and hence also satisfies ``condition (B)" and ``condition (A)".

First we consider the local case.

\subsection{The local case}
Let $v$ be a place of $F$. Consider the mapping $q_v: V(F_v)\ra F_v$  given by $q_v(x)=(x,x)$ for $x\in V(F_v)$.

For $\Phi \in \CS(V(F_v))$, define a function $F_{\Phi}^{*}$ on $F_v$ by
\[
F_{\Phi}^{*}(a)=\int_{V(F_v)}\Phi(x)\psi_v(aq_v(x))dx, \quad \forall a\in F_v.
\]
where $dx$ is the self-dual Haar measure on $V(F_v)$ with respect to $\psi_v$.

\begin{lem}
Let $C$ be a compact subset of $\CS(V(F_v))$. Then there exists a positive constant $c$ such that
\[
|F^*_{\Phi}(a)|\leq c\cdot \max(1, |a|_v)^{-m}
\]
for all $\Phi \in C$ and $a\in F_v$, where $|\cdot|_v$ is the absolute value on $F_v$.
\end{lem}
\begin{proof}
This follows from \cite[Lem. 1]{Haris} by noting that $V(F_v)$ is of dimension $2m\geq 6$ over $F_v$ when regarded as a quadratic space over $F_v$.
\end{proof}

\begin{prop}
The mapping $q_v: V(F_v)\ra F_v$ satisfies ``condition (A)''.
\end{prop}
\begin{proof}
By the above lemma, it suffices to show that $\int_{F_v} \max(1, |a|_v)^{-m}da$ is convergent. Now
\[\begin{aligned}
\int_{F_v} \max(1, |a|_v)^{-m}da
&=\int_{O_v}\max(1, |a|_v)^{-m}da+\int_{F_v-O_v}\max(1, |a|_v)^{-m}da\\
&=\Vol(O_v)+\int_{|a|_v>1}\max(1, |a|_v)^{-m}da\\
&=\Vol(O_v)+\int_{|a|_v>1}|a|_v^{-m}da.
\end{aligned}\]
But
\[\begin{aligned}
\int_{|a|_v>1}|a|_v^{-m}da
&=\sum_{n\geq 1}\int_{|a|_v=q_v^n} q_v^{-mn}da=\sum_{n\geq 1}q_v^{-mn}\int_{|a|_v=q_v^n}da\\
&=\Vol(O_v^{\times}, da)\sum_{n\geq 1}q_v^{n(1-m)}<\infty
\end{aligned}\]
since $m\geq 3$, where $q_v$ is the cardinality of the residue field of $O_v$.
Thus
\[
\int_{F_v} \max(1, |a|_v)^{-m}da<\infty.
\]
\end{proof}

For $\Phi \in \CS(V(F_v))$, let $F_{\Phi}$ be the Fourier transform of $F_{\Phi}^{*}$, i.e.
\[
F_{\Phi}(b)=\int_{F_v} F_{\Phi}^{*}(a)\psi(ab)da, \quad \forall b\in F_v,
\]
where $da$ is the self-dual Haar measure on $F_v$ with respect to $\psi_v$.

Since $q_v: V(F_v) \ra F_v$ satisfies ``condition (A)'',  by Proposition 2.1, for any $a\in F_v$, there is a measure $\mu_{a}$ on $V(F_v)$, of support contained in $q_v^{-1}(\{a\})$ such that
\[
F_{\Phi}(a)=\int_{V(F_v)} \Phi d\mu_{a}
\]
and
\[
\int_{F_v} F_{\Phi}(a)da=\int_{V(F_v)}\Phi(x)dx
\]
for all continuous functions $\Phi$ on $V(F_v)$ with compact support.

Since $q_{v}^{-1}(\{a\})=V_a(F_v)\cup \{0\}$ and $\mu_a(\{0\})=0$, we have
\[
 F_{\Phi}(a)=\int_{q_{v}^{-1}(\{a\})}\Phi d\mu_a=\int_{V_a(F_v)}\Phi d\mu_a.
\]
Next we identity the measures $\mu_a$. For $a\in F_v$, let $V_a(F_v)=\{x\in V(F_v): q_v(x)=a, x\neq 0\}$; then there is a gauge form $\theta_a(x)=(dx/d(q_v(x)))_a$ on $V_a(F_v)$ (see
\cite[p. 12]{Weil65}). Moreover, by the discussions on \cite[$\S 6$, p. 13]{Weil65}, the measure $|\theta_a|_v$ on $V_a(F_v)$ derived from $\theta_a$ has the following property: for all continuous functions $\Phi$ on $V(F_v)$ with compact support, we have
\[
\int_{V(F_v)}\Phi(x)dx=\int_{F_v} \left(\int_{V_a(F_v)}\Phi|\theta_a|_v \right)da.
\]

On the other hand, for such functions $\Phi$, we have
\[
\int_{V(F_v)}\Phi(x)dx=\int_{F_v} \left(\int_{V_a(F_v)}\Phi d\mu_a\right) da.
\]

By the uniqueness of the measures $\mu_{a}$, we have $\mu_{a}=|\theta_a|_v$ on $V_a(F_v)$.
\begin{lem}
For $\Phi \in \CS(V(F_v))$, the Fourier transform $F_{\Phi}$ of $F_{\Phi}^*$ is given by
\[
F_{\Phi}(a)=\int_{V_a(F_v)} \Phi |\theta_a|_v, \quad \forall a\in F_v.
\]
\end{lem}
\begin{proof}
We have seen that the equality holds for $\Phi$ continuous and of compact support. Since the space of all the continuous functions $\Phi$ on $V(F_v)$ with compact support is dense in $\CS(V(F_v))$, the equality holds for all $\Phi \in \CS(V(F_v))$.
\end{proof}

\subsection{The global case}
In this section, we show that the mapping $q_{\BA}: V(\BA) \ra \BA$ given by $q_{\BA}(x)=(x,x)$ satisfies ``condition (B)" and ``condition (A)".

Recall that for $\Phi \in \CS(V(\BA))$, we have defined a function $F_{\Phi}^{*}$ on $\BA$ by
\[
F_{\Phi}^{*}(a)=\int_{V(\BA)}\Phi(x)\psi(aq_{\BA}(x))dx.
\]
\begin{lem}
The mapping $q_{\BA}$ satisfies ``condition $(B_0)$", i.e. $(\Phi, a)\mapsto \Phi_a$ is a continuous mapping of $\CS(V(\BA))\times \BA$ into $\CS(V(\BA))$, where $\Phi_a(x)=\Phi(x)\psi(a q_{\BA}(x))$.
\end{lem}
\begin{proof}
This is obvious.
\end{proof}

\begin{lem}
The mapping $q_{\BA}$ satisfies ``condition $(B_1)$'', i.e. the series $\sum_{a\in F}|F_{\Phi}^{*}(a)|$ is uniformly convergent on every compact subset of $\CS(V(\BA))$.
\end{lem}
\begin{proof}
We follow the arguments on \cite[p. 230]{Haris} (see also \cite[p. 190]{Igusa}).

Let $C$ be a compact subset of $\CS(V(\BA))$. We want to show that the series $\sum_{a\in F}|F_{\Phi}^{*}(a)|$ is uniformly convergent on $C$.

Note that the adelic space $V(\BA)$ is the inductive limit of $V_T=V_T^0\times V_T^1$, where $V_T^0=\prod_{v\notin T} V_v^0$ and $V_T^1=\prod_{v\in T}V(F_v)$, as $T$ runs over all the finite sets of places of $F$, here $V_v^0=V(O_v)$.

For the compact subset $C$ of $\CS(V(\BA))$, there exists a finite set $S$ of places of $F$ and a compact subset $C_1$ of $\CS(V_S^1)$, such that every $\Phi \in C$ is of the form $\Phi=\Phi_0\otimes \Phi_1$, where $\Phi_0$ is the characteristic function of $V_S^0=\prod_{v\notin S}V_v^0$ and $\Phi_1\in C_1$.

By Lemma 2.3 and Fubini's theorem, there is a positive constant $c$ such that
\[
\sum_{a\in F}|F_{\Phi}^{*}(a)|\leq c \sum_{a\in F}\prod_v \max(1, |a|_v)^{-m}
\]
for all $\Phi \in C$.
By \cite[Prop. 1]{Haris}, the right hand side is convergent since $m\geq 3$. The desired result follows.
\end{proof}

Consequently, by Proposition 2.2, we have
\begin{prop}
The mapping $q_{\BA}: V(\BA)\ra \BA$ satisfies ``condition (B)'' and ``condition (A)''.
\end{prop}

Let $F_{\Phi}$ be the Fourier transform of $F_{\Phi}^{*}$. Then by Proposition 2.2, we have the following Poisson formula.
\begin{prop}
For $\Phi \in \CS(V(\BA))$, we have
\[
\sum_{a\in F}F_{\Phi}^{*}(a)=\sum_{a\in F}F_{\Phi}(a),
\]
and both series are absolutely convergent.
\end{prop}

Now we show that $\sum_{a\in F}F_{\Phi}^{*}(a)$ is related to the Siegel Eisenstein series $E(1,s_0,\Phi)$, where $1$ is the identity element in the group $G(\BA)$.

\begin{prop}
For $\Phi \in \CS(V(\BA))$, we have
\[
E(1,s_0,\Phi)=\Phi(0)+\sum_{a\in F}F_{\Phi}^{*}(a),
\]
and the series on the right hand side is absolutely convergent.
\end{prop}
\begin{proof}
By Bruhat decomposition $G(F)=P(F)\bigcup P(F)wN(F)$, we have
\[
E(1,s_0,\Phi)=\sum_{\gamma \in P(F)\bs G(F)}f_{\Phi}^{(s_0)}(\gamma)=f_{\Phi}^{(s_0)}(1)+\sum_{b\in F}f_{\Phi}^{(s_0)}(wn(b)).
\]
Note that $f_{\Phi}^{(s_0)}(1)=\Phi(0)$ and $f_{\Phi}^{(s_0)}(wn(b))=F_{\Phi}^{*}(b)$.
So
\[
E(1,s_0,\Phi)=\Phi(0)+\sum_{b\in F}F_{\Phi}^{*}(b).
\]
The absolute convergence follows from the above proposition.
\end{proof}

\begin{lem}
For any $\Phi \in \CS(V(\BA))$ and $g\in G(\BA)$, the Siegel Eisenstein series $E(g,s,\Phi)$ is holomorphic at $s_0$.
\end{lem}
\begin{proof}
Note that $E(g,s_0,\Phi)= E(1,s_0,\omega(g)\Phi)$ is absolutely convergent.

In general, $E(g,s,\Phi)=f_{\Phi}^{(s)}(g)+\sum_{b\in F}f_{\Phi}^{(s)}(wn(b)g)$ is a series of holomorphic functions in $s$. To show that $E(g,s,\Phi)$ is holomorphic at $s_0$, by a theorem of Weierstrass (\cite[p. 177]{Ahlfors}), it suffices to show that the series converges uniformly on every compact neighborhood of $s_0$. If $|s-s_0|\leq r$, then
\[
\sum_{b\in F}|f_{\Phi}^{(s)}(wn(b)g)|\leq \sum_{b\in F}|a(wn(b)g)|^{r}|\omega(wn(b)g)\Phi(0)|.
\]
Similar to \cite[Lem. 6.7]{Knapp}, we can show that there exists a positive constant $C=C(g)$ such that \[
|a(wn(b)g)|\leq C|a(g)|
\]
for all $b\in F$ . So for $|s-s_0|\leq r$ we have
\[
\sum_{b\in F}|f_{\Phi}^{(s)}(wn(b)g)|\leq C^r |a(g)|^r \sum_{b\in F} |\omega(wn(b)g)\Phi(0)|.
\]
Note that $\sum_{b\in F} \omega(wn(b)g)\Phi(0)=E(g,s_0,\Phi)-\omega(g)\Phi(0)$ is absolutely convergent. It follows that $\sum_{b\in F}f_{\Phi}^{(s)}(wn(b)g)$ is uniformly convergent on the disc $|s-s_0|\leq r$ by the Weierstrass M test. Thus $E(g,s,\Phi)$ is holomorphic at $s=s_0$.
\end{proof}

Next we show that $\sum_{a\in F}F_{\Phi}(a)$ is related to the theta integral $I(1,\Phi)$.
For $a\in F$, let $V_a(F)=\{x \in V(F): x\neq 0, q(x)=a\}$. When $V_a(F)\neq \emptyset$, choose $\xi_a\in V_a(F)$ and let $H_a$ be the stabilizer of $H$ at $\xi_a$, i.e. $H_a=\{h\in H: h\xi_a=\xi_a\}$. Thus $V_a$ and $H_a$ are algebraic groups over $F$. By Witt's theorem, $V_a(F)\cong H_a(F)\bs H(F)$. Moreover, by \cite[Thm. A(ii)]{BS} and \cite[Thm. 2.4.2]{Weil82}, we have $V_a(\BA)\cong H_a(\BA)\bs H(\BA)$.

We can describe the structure of $H_a$ as follows. We say a group is of type $U_n$ if it is the unitary group associated to a hermitian space of dimension $n$, and we say a group is of type $\BG_a$ if it is isomorphic to $\BG_a^r$ for some positive integer $r$.
\begin{lem}
Assume $V_a(F)\neq \emptyset$.

(i) If $a\neq 0$, then $H_a$ is of type $U_{m-1}$.

(ii) If $a=0$, then $H_a$ is isomorphic to the semidirect product of a group of type $U_{m-2}$ and a group $H'$, where $H'$ is the semidirect product of two groups of type $\BG_a$.
\end{lem}
\begin{proof}
Similar to the proof of \cite[Lem. 4.1.1]{Weil82}.
\end{proof}

\begin{cor}
Assume $V_a(F)\neq \emptyset$. Then $(1)$ is a system of convergence factors for $V_a$. Moreover, $\tau(H)=\tau(H_a)=2$ and $\tau(V_a)=1$.
\end{cor}
\begin{proof}
First note that $(1)$ is a system of convergence factors for a group of type $\BG_a$, so $H_a$ has the same system of convergence factors as $H$ and $(1)$ is a system of convergence factors for $V_a\cong H_a\bs H$ by \cite[Thm. 2.4.3]{Weil82}.

Next note that the Tamagawa number of the unitary $H=U(V)$ is equal to 2: we know that $\tau(SU(V))=1$ (see \cite[Thm. 4.4.1]{Weil82}) and $\tau(E^1)=2$ (see \cite[p. 65]{Weil82} and \cite[Definition 4.7]{Oes}), where $E^1=\{a\in E: a\bar{a}=1\}$, by noting that the Artin $L$-function $L(s, E^1)=L(s,\epsilon)$; so $\tau(H)=\tau(SU(V)) \cdot \tau(E^1)=2$ by \cite[Corollary to Thm. 2.4.4]{Weil82}. Also note that the Tamagawa number of a group of type $\BG_a$ is equal to 1, so the Tamagawa number of $H_a$ is the same as that of a unitary group by \cite[Corollary to Thm. 2.4.4]{Weil82} and hence $\tau(H_a)=2$.
\end{proof}

Next we study $F_{\Phi}$. Since $q_{\BA}: V(\BA) \ra \BA$ satisfies ``condition (A)'', by Proposition 2.1, for each $a\in F$, there is a positive measure $\mu_a$ on $V(\BA)$, of support contained in $q_{\BA}^{-1}(\{a\})$, such that
\[
F_{\Phi}(a)=\int_{V(\BA)} \Phi d\mu_a.
\]
Since $q_{\BA}^{-1}(\{a\})=V_a(\BA)\cup \{0\}$ and $\mu_a(\{0\})=0$, we have
\[
 F_{\Phi}(a)=\int_{q_{\BA}^{-1}(\{a\})}\Phi d\mu_a=\int_{V_a(\BA)}\Phi d\mu_a.
\]
Note that if $V_a(F)=\emptyset$, then $V_a(\BA)=\emptyset$ by the Hasse-Minkowski principle for quadratic forms (see \cite[p. 223]{Scharlau} or \cite[p. 170]{Lam}), and hence $F_{\Phi}(a)=0$.

Next we identify the measures $\mu_a$. For $a\in F$, consider the gauge form $\theta_a(x)=(dx/d(q(x)))_a$ on $V_a$ (see \cite[p. 12]{Weil65}). Let $|\theta_a|_{\BA}$ be the Tamagawa measure on $V_a(\BA)$ derived from $\theta_a$.
\begin{lem}
For all $\Phi \in \CS(V(\BA))$, we have
\[
F_{\Phi}(a)=\int_{V_a(\BA)}\Phi |\theta_a|_{\BA}, \quad \forall a\in F.
\]
\end{lem}
\begin{proof}
We follow the arguments in the proof of \cite[Lem. 3]{Haris} (see also \cite[p. 191]{Igusa}).

It suffices to show the equality for $\Phi$ restricted to a subset of $\CS(V(\BA))$ which spans a dense subset of $\CS(V(\BA))$.

Take $\Phi=\prod_v \Phi_v$, where each $\Phi_v \in \CS(V(F_v))$ and $\Phi_v$ is the characteristic function of $V(O_v)$ for almost all $v$. Then $F_{\Phi}(a)=\prod_v F_{\Phi_v}(a)$,
where $F_{\Phi_v}$ is the Fourier transform of $F^{*}_{\Phi_v}$.
By Lemma 2.5, $F_{\Phi_v}(a)=\int_{V_a(F_v)} \Phi_v |\theta_a|_v$. Moreover, since $(1)$ is a system of convergence factors for $V_a$, $|\theta_a|_{\BA}=\prod_v |\theta_a|_v$. Thus
\[
F_{\Phi}(a)=\prod_v F_{\Phi_v}(a)=\prod_v \int_{V_a(F_v)} \Phi_v |\theta_a|_v=\int_{V_a(\BA)} \Phi |\theta_a|_{\BA}.
\]
\end{proof}

For $a\in F$, let $D_a$ be the gauge form on $V_a$ which is the image of $\frac{\eta}{\eta_a}$ under the isomorphism $V_a\cong H_a\bs H$, where $\eta$ (resp. $\eta_a$) is a left invariant gauge form on $H$ (resp. on $H_a$). Let $|D_a|_{\BA}$ be the Tamagawa measure on $V_a(\BA)$ derived from $D_a$.
\begin{lem}
For $a\in F$ with $V_a(F)\neq \emptyset$, let $H_a$ be the stabilizer of $H$ at any element in $V_a(F)$. Then
\[
\int_{[H]}\left(\sum_{\xi \in V_a(F)}\Phi(h^{-1}\xi)\right)dh=\tau(H_a)\int_{V_a(\BA)}\Phi |D_a|_{\BA}.
\]
Note that if $V_a(F)=\emptyset$, then $V_a(\BA)=\emptyset$  and the two integrals above are all zero.
\end{lem}
\begin{proof}
Assume $V_a(F)\neq \emptyset$. Choose $\xi_a \in V_a(F)$. Since $H_a(F)\bs H(F)\cong V_a(F)$ via $h\mapsto h^{-1}\xi_a$, we have
\[
\int_{[H]}\left(\sum_{\xi \in V_a(F)}\Phi(h^{-1}\xi)\right)dh=\int_{H_a(F)\bs H(\BA)} \Phi(h^{-1}\xi_a)dh.
\]

Now
\[\begin{aligned}
\int_{H_a(F)\bs H(\BA)} \Phi(h^{-1}\xi_a)dh
&=\int_{H_a(\BA)\bs H(\BA)}\int_{H_a(F)\bs H_a(\BA)}\Phi(\bar{h}^{-1}h_a^{-1}\xi_a)d\bar{h}~dh_a \\
&=\tau(H_a)\int_{H_a(\BA)\bs H(\BA)}\Phi(\bar{h}^{-1}\xi_a)d\bar{h}\\
&=\tau(H_a)\int_{V_a(\BA)}\Phi |D_a|_{\BA},
\end{aligned}\]
where $dh_a$ is the Tamagawa measure on $H_a(\BA)$, $d\bar{h}=\frac{dh}{dh_a}$ is the quotient measure on $H_a(\BA)\bs H(\BA)$, and note that $|D_a|_{\BA}$ is the image measure of $d\bar{h}$ under the isomorphism $V_a(\BA)\cong H_a(\BA)\bs H(\BA)$.
\end{proof}

Note that $D_a$ and $\theta_a$ are both left invariant gauge forms on $V_a$, so $|D_a|_{\BA}=|\theta_a|_{\BA}$.
\begin{lem}
For $a\in F$, we have
\[
\int_{[H]}\left(\sum_{\xi \in V_a(F)}\Phi(h^{-1}\xi)\right)dh=2F_{\Phi}(a).
\]
\end{lem}

By the discussions in $\S 1$, we thus have the following result.
\begin{prop}
For $\Phi \in \CS(V(\BA))$, we have
\[
I(1,\Phi)=2\Phi(0)+2\sum_{a\in F}F_{\Phi}(a),
\]
and it is absolutely convergent.
\end{prop}

Now we can finally show the Siegel-Weil formula.

\begin{thm}
Assume $m\geq 3$. Then for any $\Phi \in \CS(V(\BA))$ and $g\in G(\BA)$, $E(g,s,\Phi)$ is holomorphic at $s_0=(m-1)/2$, $I(g, \Phi)$ is absolutely convergent, and
\[
I(g,\Phi)=2E(g, s_0, \Phi).
\]
\end{thm}
\begin{proof}
First note that $E(g,s,\Phi)$ is holomorphic at $s_0=(m-1)/2$ by Lemma 2.11.
Also note that $E(g,s_0,\Phi)=E(1,s_0,\omega(g)\Phi)$ and $I(g,\Phi)=I(1,\omega(g)\Phi)$, so it suffices to consider the case $g=1$.

The equality $I(1,\Phi)=2E(1,s_0,\Phi)$ follows from propositions 2.9, 2.10 and 2.17.
\end{proof}

\bibliographystyle{amsplain}

\end{document}